\documentclass[11pt,letterpaper,reqno]{amsart}
\usepackage{tikz}
\usetikzlibrary{positioning, shapes.geometric, arrows.meta}
\usepackage{amssymb}
\usepackage{amsmath}
\usepackage{amsthm}
\usepackage{amsfonts}
\usepackage{bbm}
\usepackage{enumitem} 
\usepackage{pgfplots}
\pgfplotsset{compat=1.18} 
\usepackage{booktabs}

\usepackage{graphicx}
\usepackage[T1]{fontenc}
\usepackage{doi}
\usepackage{float} 
\usepackage{bookmark}
\usepackage{hyperref}
\hypersetup{pdfstartview={FitH}}

\newtheorem{thm}{Theorem}[section]
\newtheorem{lem}[thm]{Lemma}

\newtheorem{prob}[thm]{Problem}

\theoremstyle{definition}

\numberwithin{equation}{section}

\makeatother

\begin{document}

\title{An improved lower bound for Erd\H{o}s--Szekeres products}

\author[Q.~Tang]{Quanyu Tang}

\address{School of Mathematics and Statistics, Xi'an Jiaotong University, Xi'an 710049, P. R. China}
\email{tang\_quanyu@163.com}



\begin{abstract}
In 1959, Erd\H{o}s and Szekeres posed a series of problems concerning the size of polynomials of the form 
$$
P_n(z) = \prod_{j=1}^n (1 - z^{s_j}),
$$
where $s_1, \dots, s_n$ are positive integers. Of particular interest is the quantity 
$$
f(n) = \inf_{s_1,\dots,s_n\ge 1}  \max_{|z|=1} |P_n(z)|.
$$They proved that $\lim_{n\to\infty} f(n)^{1/n} = 1$, and also established the classical lower bound $f(n) \ge \sqrt{2n}$. However, despite extensive effort over more than six decades, no stronger general lower bound had been established. 

In this paper, we obtain the new bound
$$
f(n) \ge 2\sqrt{n}.
$$This gives the first improvement of the classical lower bound for the Erd\H{o}s--Szekeres problem in the general case since 1959. In particular, our result confirms a remark of Billsborough et al.~\cite{BFHKLPS21}, who observed that if the original Erd\H{o}s--Szekeres proof could be fixed, the O’Hara--Rodriguez bound would yield exactly this inequality.
\end{abstract}

\maketitle

\section{Introduction}

In their short but influential 1959 paper~\cite{ErdosSzekeres1959}, Erd\H{o}s and Szekeres posed a number of 
problems about the growth or decay of ``pure power products''
\begin{equation}\label{eq:ESproduct}
P_n(z) = \prod_{j=1}^n (1 - z^{s_j})=\sum_{k=1}^N a_k z^k,
\end{equation}
and their norms $\| P_n \|_{L^\infty(|z|=1)}$, where $s_1,\dots,s_n$ are positive integers. Here $\|P_n\|_{L^\infty(|z|=1)}$ denotes the maximum modulus of $P_n$ on the unit circle, namely $\|P_n\|_{L^\infty(|z|=1)} = \max_{|z|=1} |P_n(z)|$. For an excellent survey of these problems and related developments, see Lubinsky \cite{LubinskySurvey}. 

Among the many problems raised by Erd\H{o}s and Szekeres, the most central one is the following, 
as highlighted in \cite[Problem~1]{LubinskySurvey} and also listed as Problem~\#256 on Bloom's Erd\H{o}s Problems website \cite{TFB}:

\begin{prob}\label{problem1}
Define
\[
f(n) := \inf_{s_1,\dots,s_n \ge 1} \; \| P_n \|_{L^\infty(|z|=1)}.
\]
Determine the growth of $f(n)$ as $n \to \infty$.    
\end{prob}

Erd\H{o}s and Szekeres~\cite{ErdosSzekeres1959} showed that $\lim_{n\to\infty} f(n)^{1/n} = 1$, and suggested that $f(n)$ should grow faster than any power of $n$ \cite{Erdos1964}. Later progress includes upper bounds by Atkinson \cite{Atkinson1961}, Odlyzko \cite{Odlyzko1982}, Kolountzakis \cite{Kolountzakis1994}, and the deep result of Belov and Konyagin \cite{BelovKonyagin1996} giving
\[
\log f(n) = O\left((\log n)^4\right).
\]

On the lower bound side, the classical result $f(n) \ge \sqrt{2n}$ goes back to Erd\H{o}s and Szekeres~\cite{ErdosSzekeres1959}. 
Although their original proof contained a gap, this was later repaired in~\cite{BFHKLPS21} 
(see also Section~\ref{sec:classical} for details). There have also been several results that treat the sequence $\{s_j\}$ under additional restrictions; see \cite[pp.~300--302]{LubinskySurvey} for a detailed account, of which we give only a brief summary here. 
Borwein~\cite{Borwein1993} obtained a new lower bound in the case when none of the $\{s_j\}$ are divisible by a given prime $p$. 
Bell, Borwein, and Richmond~\cite{BBR1998} obtained both asymptotic upper and lower bounds. 
Bourgain and Chang~\cite{BC18}, using probabilistic methods and careful estimation of trigonometric sums, established new results for “dense” sets $S = \{s_1 < \cdots < s_n\}$. 
In 2021, Billsborough et al.~\cite[Theorem~2.1]{BFHKLPS21} obtained new lower bounds using the Poisson integral of $\log|P_n|$, 
which are effective when the $\{s_j\}$ do not grow too quickly. 
Another result in the same paper, \cite[Theorem~2.2]{BFHKLPS21}, proved using Kellogg’s extension of the Hausdorff--Young 
inequalities for Fourier coefficients, gives good bounds when the $\{s_j\}$ grow rapidly or are well separated. 
Nevertheless, as emphasized in \cite[p.~300]{LubinskySurvey} (see also \cite[p.~253]{BC18} or \cite{BFHKLPS21}), there has been no improvement on the lower bound $f(n)\ge \sqrt{2n}$ in general.


Our main result is the following theorem:

\begin{thm}\label{thm:main}
For all $n \in \mathbb{N}$,
\[
f(n) \;\ge\; 2\sqrt{n}.
\]
\end{thm}

Thus we provide the first improvement of the classical lower bound in the Erd\H{o}s--Szekeres problem that holds in the general case since 1959.

The rest of this paper is organized as follows. 
In Section~\ref{sec:classical}, we revisit the classical lower bound $f(n)\ge \sqrt{2n}$. 
This section not only reviews the original argument and its correction in later works, 
but also provides the motivation for our proof of Theorem~\ref{thm:main}. 
In Section~\ref{sec:main_result}, we establish Theorem~\ref{thm:L2-lower-bound}, 
which may be of independent interest since it depends only on the divisibility condition $(1-z)^n \mid p(z)$. 
Finally, we prove our main result, Theorem~\ref{thm:main}.


\section{The classical lower bound}\label{sec:classical}

In this section we briefly recall the classical proof of the lower bound
\[
f(n)\ge\sqrt{2n},
\]following the exposition in \cite[Section~3]{BFHKLPS21}. This will make it clear how our new approach strengthens the existing argument.

\subsection{The original Erd\H{o}s--Szekeres idea}
Erd\H{o}s and Szekeres~\cite{ErdosSzekeres1959} considered the identity
\begin{equation}\label{eq:ESidentity}
\sum_{j=1}^r z^{a_j} - \sum_{j=1}^r z^{b_j} \;=\; \prod_{j=1}^n (1-z^{s_j}),
\end{equation}
where $\{a_j\}$ and $\{b_j\}$ are disjoint sets of integers. In this formulation all nonzero coefficients on the left-hand side are $\pm 1$. Since the right-hand side of~\eqref{eq:ESidentity} has a zero of multiplicity $n$ at $z=1$, differentiating the left-hand side of~\eqref{eq:ESidentity} $k$ times and then setting $z=1$ gives
\[
\sum_{j=1}^r a_j^k \;=\; \sum_{j=1}^r b_j^k, \qquad k=0,1,\dots,n-1.
\]
By the classical Prouhet--Tarry--Escott problem~\cite{BI94}, this forces $r\ge n$. Hence there are at least $2n$ nonzero coefficients in $P_n(z)$, and therefore\[
\sum_{k=0}^N |a_k|^2 \;\ge\; 2n.
\]Finally, applying Parseval’s identity we obtain
\[
\| P_n \|_{L^\infty(|z|=1)}^2 \geq \frac{1}{2\pi}\int_{-\pi}^{\pi} |P_n(e^{it})|^2\,dt
= \sum_{k=0}^N |a_k|^2 \ge 2n,
\]which in turn implies $\| P_n \|_{L^\infty(|z|=1)} \ge \sqrt{2n}$. Thus one arrives at the classical bound $f(n) \ge \sqrt{2n}$.

\subsection{Problems with the argument}
However, as observed in \cite{BFHKLPS21}, the identity \eqref{eq:ESidentity} does not hold for general $n$, since the coefficients of $\prod_{j=1}^n (1-z^{s_j})$ need not all be $\pm1$. 
For example,
\[
(1-z)^n = \sum_{k=0}^n (-1)^k \binom{n}{k} z^k
\]
has only $n+1$ nonzero coefficients, not $2n$. 
Thus the conclusion that there are always at least $2n$ nonzero coefficients is not justified in this form.

\subsection{A corrected proof}
A rigorous proof was later given in \cite[Section~3]{BFHKLPS21}. Since $(1-z)^n \mid P_n(z)$, we have
\[
P_n^{(j)}(1)=0, \qquad j=0,1,\dots,n-1.
\]Equivalently,\[
\sum_{k=0}^N a_k (k)_j = 0, \qquad j=0,1,\dots,n-1.
\]This implies that\[
\sum_{k=0}^N a_k S(k)=0
\]for every polynomial $S(x)$ of degree at most $n-1$. As $P_n$ is not the zero polynomial, this forces at least $n$ coefficients $a_k$ to be nonzero.\footnote{In fact, with a more refined argument one can show that $P_n$ has at least $n+1$ nonzero coefficients.} Hence \(\sum_{k=0}^N |a_k|^2 \ge n\). Finally, by a result of O’Hara and Rodriguez~\cite[Corollary~1]{OR74}, we have
\[
\|P_n\|_{L^\infty(|z|=1)}^2 \ge 2\sum_{k=0}^N |a_k|^2 \ge 2n,
\]and therefore \(f(n) \ge \sqrt{2n}\).

This is the classical lower bound, rigorously justified in \cite{BFHKLPS21}. Interestingly, in \cite{BFHKLPS21} (see also \cite[p.~298]{LubinskySurvey}) the authors further remarked:
\begin{quote}
\emph{If the original Erd\H{o}s--Szekeres proof could be fixed, the O’Hara--Rodriguez bound would give $f(n)\ge 2\sqrt{n}$.}
\end{quote}Our present work may be viewed as providing precisely such a fix, thereby confirming their remark. Indeed, the proof of Theorem~\ref{thm:main} also makes use of the O’Hara--Rodriguez bound. The key new ingredient, however, is that in Theorem~\ref{thm:L2-lower-bound} we prove the stronger estimate
\[
\sum_{k=0}^N |a_k|^2 \ge 2n,
\]
which in turn yields the improved inequality $f(n)\ge 2\sqrt{n}$.

\section{Proof of the Main Result}\label{sec:main_result}
We first prove a structural lemma concerning equality of power sums, based on Newton’s identities. Here by an \emph{integer multiset}\footnote{A \emph{multiset} (also called a \emph{bag}) is like a set, except that elements are allowed to appear with multiplicities. For instance, $\{1,1,2,3\}$ is a multiset that differs from the set $\{1,2,3\}$.} we mean a collection of integers in which elements may appear with multiplicity.

\begin{lem}\label{lem:PS-unique}
Let $X=\{x_1,\dots,x_m\}$ and $Y=\{y_1,\dots,y_m\}$ be integer multisets. 
If
\[
\sum_{i=1}^m x_i^r \;=\; \sum_{i=1}^m y_i^r \qquad\text{for } r=1,\dots,m,
\]
then $X$ and $Y$ are identical as multisets.
\end{lem}

\begin{proof}
Let $p_r(Z):=\sum_{z\in Z} z^r$ be the $r$th power sum and 
$e_j(Z)$ the $j$th elementary symmetric polynomial in the elements of $Z$ ($e_0(Z):=1$).
For any $m$-tuple $Z$, the Newton's identities assert that for $r=1,\dots,m$,
\[
p_r(Z) - e_1(Z)p_{r-1}(Z) + e_2(Z)p_{r-2}(Z) - \cdots 
+ (-1)^{r-1} e_{r-1}(Z)p_1(Z) + (-1)^r r\, e_r(Z) = 0,
\]
and in particular give the recursion
\[
e_r(Z)=\frac{1}{r}\left(p_r(Z) - e_1(Z)p_{r-1}(Z) + \cdots 
+ (-1)^{r-1} e_{r-1}(Z)p_1(Z)\right), \qquad r=1,\dots,m.
\]
Hence the vector $(p_1(Z),\dots,p_m(Z))$ uniquely determines $(e_1(Z),\dots,e_m(Z))$.

Applying this to $Z=X$ and $Z=Y$ and using the hypothesis $p_r(X)=p_r(Y)$ for $r=1,\dots,m$, 
we obtain $e_j(X)=e_j(Y)$ for all $j=1,\dots,m$. 
Therefore the associated monic polynomials
\[
P_X(t):=\prod_{x\in X}(t-x)=t^m-e_1(X)t^{m-1}+\cdots+(-1)^m e_m(X),
\]
and $P_Y(t)$ have the same coefficients, so $P_X\equiv P_Y$. 
Thus $X$ and $Y$ have the same roots with multiplicities, i.e., they are identical as multisets.
\end{proof}

\begin{thm}\label{thm:L2-lower-bound}
Let $n\in\mathbb{N}$ and let \(p(z)=\sum_{k=1}^N a_k z^k \in \mathbb{Z}[z]\) be a nonzero polynomial such that $(1-z)^n$ divides $p(z)$. Then
\[
\sum_{k=1}^N |a_k|^2 \ge 2n.
\]
\end{thm}
\begin{proof}
Let 
\[
p(z)=\sum_{k=1}^N a_k z^k \in \mathbb{Z}[z]
\]
be a nonzero polynomial. By definition, $(1-z)^n \mid p(z)$ iff $p^{(r)}(1)=0$ for $r=0,1,\dots,n-1$. Since
\[
p^{(r)}(z)=\sum_{k=1}^N a_k \frac{d^r}{dz^r}z^k
=\sum_{k=1}^N a_k (k)_r z^{k-r},
\]
we obtain, upon evaluating at $z=1$,
\[
p^{(r)}(1)=\sum_{k=1}^N a_k (k)_r=0, \qquad r=0,1,\dots,n-1.
\] Since the families $\{k^r\}_{r=0}^{n-1}$ and $\{(k)_r\}_{r=0}^{n-1}$ span the same space, the simultaneous vanishing of the factorial moments $\sum_k a_k (k)_r = 0$ for $r=0,1,\dots,n-1$ is equivalent to the simultaneous vanishing of the power moments $\sum_k a_k k^r = 0$ for $r=0,1,\dots,n-1$. Therefore, we have\begin{equation}\label{eq:moment}
\sum_{k=1}^N a_k k^r = 0, \qquad r=0,1,\dots,n-1.
\end{equation}

To analyze these relations, we split the coefficients into their positive and negative parts. For each $k$, if $a_k>0$ insert $k$ into a multiset $X$ with multiplicity $a_k$; if $a_k<0$ insert $k$ into a multiset $Y$ with multiplicity $|a_k|$. Formally,
\[
X:=\{\,\underbrace{k,\dots,k}_{a_k \text{ times}} : a_k>0\,\}, 
\qquad
Y:=\{\,\underbrace{k,\dots,k}_{|a_k| \text{ times}} : a_k<0\,\}.
\]
Thus every element of $X$ or $Y$ is an index of $p(z)$, with multiplicity determined by the absolute value of the corresponding coefficient. With this notation, Equation \eqref{eq:moment} can now be rewritten as
\[
\sum_{x\in X} x^r = \sum_{y\in Y} y^r, \qquad r=0,1,\dots,n-1.
\]Let $m=|X|=|Y|$. If $m\le n-1$, then the power sums agree for $r=1,\dots,m$. By Lemma~\ref{lem:PS-unique}, this forces $X=Y$, hence all $a_k=0$, contradicting that $p\neq 0$. Therefore,
\begin{equation}\label{eq:m-lower}
m \ge n.
\end{equation}

Finally, notice that
\[
\sum_{k=1}^N |a_k| = |X|+|Y| = 2m \ge 2n \qquad \text{(by \eqref{eq:m-lower})}.
\]
Moreover, for each nonzero integer $a_k$, we clearly have $|a_k|^2 \ge |a_k|$.  
Summing over all $k$ gives
\[
\sum_{k=1}^N |a_k|^2 \ge \sum_{k=1}^N |a_k| \ge 2n.
\]This completes the proof.
\end{proof}

Before proving Theorem~\ref{thm:main}, we make some remarks on Theorem~\ref{thm:L2-lower-bound}. 
In fact, Theorem~\ref{thm:L2-lower-bound} establishes a stronger conclusion: 
for any nonzero integer polynomial $p(z)=\sum_{k} a_k z^k$ divisible by $(1-z)^n$, one has \(\sum_{k} |a_k|\ge 2n\). This can be viewed as an estimate for the $\ell^1$ norm of the coefficients, \(\|p\|_1 := \sum_k |a_k|\).

For the polynomials of the form \eqref{eq:ESproduct}, Maltby~\cite{Maltby971, Maltby972} obtained some improvements for specific values of $n$, 
focusing on the $\ell^1$ norm of the coefficients and employing algorithms to search for optimal choices of $\{s_j\}$. 
It is worth noting that in \cite[p.~1323]{Maltby972} Maltby pointed out the inequality $\|P_n\|_1 \ge 2n$, 
referring to \cite{BI94} for a proof. However, this statement is limited in scope: just as in the gap in Erd\H{o}s and Szekeres~\cite{ErdosSzekeres1959}, \cite[Section~4.2]{BI94} applied the Prouhet--Tarry--Escott problem, which is only valid for the special form (in \cite[Section~4.2]{BI94} also referred to as a pure product)
\[
\sum_{i=1}^N z^{a_i} - \sum_{i=1}^N z^{\beta_i},
\]so that all nonzero coefficients are $\pm 1$. In general this is not the case, since the coefficients of $P_n(z)$ need not be restricted to $\pm1$. 
Thus, prior to the present paper, neither Theorem~\ref{thm:L2-lower-bound} nor Theorem~\ref{thm:main} had been established.

We are now ready to prove Theorem~\ref{thm:main}. 
\begin{proof}[Proof of Theorem~\ref{thm:main}]
Let
\[
P_n(z) = \prod_{j=1}^n (1 - z^{s_j}) = \sum_{k=1}^N a_k z^k, \qquad s_j \in \mathbb{N}.
\]
Clearly $(1-z)^n \mid P_n(z)$. By Theorem~\ref{thm:L2-lower-bound}, we have
\[
\sum_{k=1}^N |a_k|^2 \ge 2n.
\]
Moreover, by \cite[Corollary~1]{OR74}, one knows that
\[
\| P_n \|_{L^\infty(|z|=1)}^2 \ge 2\sum_{k=1}^N |a_k|^2 \ge 4n,
\]since all zeros of $P_n$ lie on the unit circle. Hence
\[
\| P_n \|_{L^\infty(|z|=1)} \ge 2\sqrt{n}.
\]
Taking the infimum over all choices of $s_j$ yields the desired bound
\[
f(n)\ge 2\sqrt{n}.
\qedhere\]
\end{proof}

\section*{Acknowledgements}
The author is grateful to Tao Hu, Jiachen Shi, Binyan Shi, Hao Zhang, and Jianxiang Zhou for valuable discussions. The author also thanks Thomas Bloom for founding and maintaining the Erd\H{o}s Problems website, which has been a valuable resource and inspiration for this work.



\end{document}